\documentclass[11pt]{article}

\usepackage{mathpazo}
\usepackage{amsfonts}
\usepackage{amsmath}
\usepackage{graphicx}
\usepackage{latexsym}
\usepackage{mathabx}
\usepackage{bm}
\usepackage{longtable}
\usepackage{booktabs}
\usepackage{makecell}

\usepackage[margin=1.05in]{geometry}
\usepackage[T1]{fontenc}
\usepackage{times}
\usepackage{color,graphicx}
\usepackage{array}
\usepackage{enumerate}
\usepackage{amsmath}
\usepackage{amssymb}
\usepackage{amsthm}
\usepackage{pgfplots}
\usepackage{pgf}
\usepackage{tikz}
\usetikzlibrary{patterns}
\usepgfplotslibrary{patchplots} 
\usetikzlibrary{decorations.markings}
\usetikzlibrary{pgfplots.patchplots} 
\pgfplotsset{width=9cm,compat=1.12}
    \usepgfplotslibrary{colormaps,external}
    \definecolor{ocre}{RGB}{0,96,128}
    \definecolor{purp}{RGB}{112,0,112}
    \pgfplotsset{
    	colormap={ocrefade}{
    		rgb255=(150,216,255)
    		rgb255=(0,64,96)
    	}
    }
\usepackage{amsthm}
\newtheorem{conjecture}{Conjecture}

\newtheorem{definition}{Definition}
\newtheorem{lemma}{Lemma}

\newtheorem{theorem}{Theorem}
\newtheorem{corollary}{Corollary}

\newtheorem{proposition}{Proposition}

\newtheorem{question}{Question}
\newtheorem{example}{Example}

\usepackage{xcolor}
\usepackage{makeidx}
\usepackage[colorlinks=true,linkcolor=blue,anchorcolor=blue,citecolor=red,urlcolor=magenta]{hyperref}
\usepackage{caption}
\newcommand{\tr}{\operatorname{Tr}}
\newcommand{\rank}{\operatorname{rank}}

\newcommand{\defeq}{\stackrel{\smash{\textnormal{\tiny def}}}{=}}

\def\M{\mathcal{M}}
\def\C{\mathbb{C}}
\def\R{\mathbb{R}}

\def\v{\mathbf{v}}

\begin{document}
\title{The Factor Width Rank of a Matrix}

\author{
	Nathaniel Johnston,\textsuperscript{\!\!1} \ Shirin Moein,\textsuperscript{\!\!1,2,3} \ and Sarah Plosker\textsuperscript{3}
}

\maketitle

\begin{abstract}
    A matrix is said to have factor width at most $k$ if it can be written as a sum of positive semidefinite matrices that are non-zero only in a single $k \times k$ principal submatrix. We explore the ``factor-width-$k$ rank'' of a matrix, which is the minimum number of rank-$1$ matrices that can be used in such a factor-width-at-most-$k$ decomposition. We show that the factor width rank of a banded or arrowhead matrix equals its usual rank, but for other matrices they can differ. We also establish several bounds on the factor width rank of a matrix, including a tight connection between factor-width-$k$ rank and the $k$-clique covering number of a graph, and we discuss how the factor width and factor width rank change when taking Hadamard products and Hadamard powers.\\
    
	
\end{abstract}

\addtocounter{footnote}{1}
\footnotetext{Department of Mathematics \& Computer Science, Mount Allison University, Sackville, NB, Canada E4L 1E4}
\addtocounter{footnote}{1}
\footnotetext{Department of Mathematics, Vali-e-Asr University of Rafsanjan, P.O. Box 518, Rafsanjan, Iran}
\addtocounter{footnote}{1}
\footnotetext{Department of Mathematics \& Computer Science, Brandon University, Brandon,
    MB, Canada R7A 6A9}

\section{Introduction}

First introduced in \cite{boman2005factor}, factor width is a matrix-theoretic notion used in the study of real (symmetric) positive semidefinite matrices. Given a positive integer $k$, a matrix is said to have factor width (at most) $k$ if it can be written as a sum of outer products of vectors each having at most $k$ non-zero entries.
The factor width of the matrix is the minimum value of $k$ for which such a decomposition is possible, and it establishes a hierarchy of convex cones of real positive semidefinite matrices. The matrices with factor width $1$ are exactly the diagonal matrices with non-negative diagonal entries, and all $n \times n$ positive semidefinite matrices have factor width at most $n$.

The concept of factor width has been used in optimization theory in order to provide a more tractable alternative to the sum of squares optimization of multivariate polynomials \cite{sos}. The concept has also found use in quantum information theory, where it has  been generalized to complex (Hermitian) positive semidefinite matrices in the study of the PPT$^2$ conjecture \cite{ppt} in entanglement theory, and has also been considered in the study of superposition, where a density matrix with factor width at most $k$ is typically called ``$k$-incoherent'' \cite{levi_coherence,RBC18} (this connection between factor width and multilevel quantum coherence was made explicit in \cite{johnston2022absolutely}).

Our goal in this paper is to investigate the ``factor-width-$k$ rank'' of a matrix that has factor width at most $k$, which is the fewest number of vectors (each with at most $k$ non-zero entries) that can be used in its factor-width-$k$ decomposition. This quantity was introduced briefly in \cite{boman2005factor}, but to our knowledge has not previously been explored in any serious capacity.

Similar notions of rank have been introduced and explored for numerous other convex cones of positive semidefinite matrices, however. For example, the \emph{completely positive rank} of a completely positive matrix is the minimal number of entrywise non-negative vectors whose outer products sum to the matrix \cite{BB03}, and there is a similar notion of rank in the case when each vector has all entries belonging to the set $\{0,1\}$ \cite{berman}. In addition, the \emph{separable rank} (or equivalently, \emph{entanglement breaking rank}) arises if we restrict each vector to be an elementary tensor (or, in the language of quantum information theory, a product state) \cite{PPPR18}, and the \emph{mixed-unitary rank} arises if we restrict each vector to have all Schmidt coefficients equal to each other \cite{GL21,GLLL22}. In all of these cases, the extreme points of the convex cone are simply the rank-$1$ elements of that cone, and the notion of rank comes from asking for the fewest number of extreme points needed to express the matrix via a convex combination.

\subsection{Arrangement of the Paper}

We start in Section~\ref{sec:prelims} by going over some background material and straightforward lemmas an observations. More specifically, we recall the precise definitions of factor width and factor width rank in Section~\ref{sec:fac_wid}. We then make some observations about what effect the ground field $\R$ or $\C$ has on factor width and factor width rank in Section~\ref{sec:real_v_complex}.

In Section~\ref{sec:banded_connection}, we investigate the factor width and factor width rank of banded matrices. We show in Theorem~\ref{thm:factor_width_and_bandwidth} that a matrix's factor width is never larger than its bandwidth, which follows from a more general result (Theorem~\ref{thm:chordal_main}), which states that if the graph corresponding to a positive semidefinite matrix $A$ is chordal, then the factor width of $A$ is bounded above by the clique number of the graph. We show in Theorem~\ref{thm:tridiag_fac_two_rank} that the factor-width-$k$ rank of a banded matrix is always equal to its usual rank.

In Section~\ref{sec:factor_width_2_rank}, we explore the factor-width-$k$ rank of matrices in depth in the $k = 2$ case. We show, for example, that the factor-width-2 rank can be easily computed if a matrix has no entries equal to $0$ (Corollary~\ref{cor:fac_wid_rank2_nonzero}), is an arrowhead matrix (Theorem~\ref{thm:arrowhead_facwid2}), or is reasonably ``close'' to being block diagonal (Theorems~\ref{thm:fac_wid_rank2_3x3} and~\ref{thm:block_diag}).

In Section~\ref{sec:factor_width_rank_bounds}, we develop several lower and upper bounds on factor-width-$k$ rank that apply for all values of $k$. In particular, we establish in Section~\ref{sec:covering_design_bounds} a connection between factor-width-$k$ rank and covering numbers; quantities that seem to be difficult to compute, but have numerous known results and values that we can make use of. In Section~\ref{sec:small_mats}, we summarize what our results and bounds tell us about the factor-width-$k$ rank of small matrices. In particular, we note that we can easily compute the factor-width-$k$ rank of matrices of size $3 \times 3$ or less, and we clarify which $4 \times 4$ matrices still have unknown factor-width-$k$ rank.

Finally, in Section~\ref{sec:hadamard_product} we explore what effect the Hadamard (i.e., entrywise) product of matrices has on their factor widths and factor width ranks.

\section{Preliminaries and Definitions}\label{sec:prelims}

We denote the set of all $n\times n$ matrices with entries from either $\R$ or $\C$ by $\M_n$ (if we with to emphasize or clarify which field we are working over, we will use the notation $\M_n(\R)$ or $\M_n(\C)$ as appropriate). We similarly denote the set of positive semidefinite (PSD) matrices by $\M_n^+$, $\M_n^+(\R)$, or $\M_n^+(\C)$, and we assume symmetry of any PSD matrices herein. The transpose of a matrix $A$ is denoted by $A^T$ and its conjugate transpose is denoted by $A^*$. Every vector $\v$ that we consider is assumed to be a column vector, so $\v^T$ and $\v^*$ are row vectors. We denote the $(i,j)$-entry of matrices $A,B \in \M_n$ via the corresponding lowercase letters, as in $a_{i,j}$ and $b_{i,j}$. Alternatively, we sometimes use the notation $[A]_{i,j}$ and $[B]_{i,j}$ to refer to the same quantities, especially when we wish to discuss specific entries of more complicated matrices (e.g., the $(i,j)$-entry of $A^T+B^3$ is $[A^T+B^3]_{i,j}$).

\subsection{Factor Width and Factor Width Rank}\label{sec:fac_wid}

We now recall the definition of the factor width of a matrix from \cite{boman2005factor}:

\begin{definition}\label{defn:factor_width}
    Suppose $\mathbb{F} \in \{\R,\C\}$. The \emph{factor width} of a matrix $A \in \mathcal{M}_n^{+}(\mathbb{F})$ is the smallest positive integer $k$ with the property that $A$ can be written in the form
    \begin{align}\label{eq:fac_wid_decomp}
        A = \sum_{j=1}^r \mathbf{v}_j\mathbf{v}_j^*
    \end{align}
    for some $\mathbf{v}_1$, $\mathbf{v}_2$, $\ldots$, $\mathbf{v}_r \in \mathbb{F}^n$, each with at most $k$ non-zero entries.
\end{definition}

Thanks to the spectral decomposition, Definition~\ref{defn:factor_width} can be rephrased as saying that the factor width of a matrix $A \in \mathcal{M}_n^{+}$ is the smallest positive integer $k$ for which $A$ can be written in the form
\begin{align}\label{eq:fac_width_by_psd}
    A = \sum_j A_j,
\end{align}
where each $A_j \in \mathcal{M}_n^{+}$ equals $0$ outside of a single $k \times k$ principal submatrix (different $k \times k$ principal submatrices can be non-zero for different $A_j$'s). This formulation of factor width lets us compute it (and even find factor width decompositions) numerically via semidefinite programming \cite{NBCPJA16}. However, for some values of $k$ (e.g., if $k$ is near $n/2$) the size of this semidefinite program is exponential in $n$, so it nonetheless seems likely that computation of factor width is difficult in general.

The matrices with factor width~1 are exactly the diagonal positive semidefinite matrices, since in this case the vectors $\{\mathbf{v}_j\}$ in Definition~\ref{defn:factor_width} must be scalar multiples of standard basis vectors. At the other extreme, if $k = n$ then every $n \times n$ positive semidefinite matrix has a decomposition of the form~\eqref{eq:fac_wid_decomp}, so every PSD matrix has factor width at most $n$. As $k$ increases from $1$ to $n$, the sets of matrices with factor width at most $k$ form a nested sequence of closed convex cones that strictly contain each other.

One can further ask the question of how many terms are needed in a factor-width-$k$ decomposition (in the sense of Equation~\eqref{eq:fac_wid_decomp}) of a matrix:

\begin{definition}\label{defn:rank_of_fac_wid_decomp}
    Suppose $\mathbb{F} \in \{\R,\C\}$ and $A \in \mathcal{M}_n^{+}(\mathbb{F})$ has factor width less than or equal to $k$. We say that the \emph{factor-width-$k$ rank} of $A$, denoted by $\mathrm{fran}_k(A)$ (or $\mathrm{fran}_k^{\mathbb{F}}(A)$ if we wish to emphasize the field we are working over), is the least integer $r$ for which there exists a decomposition of the form
    \begin{align}\label{eq:fac_wid_decomp_rank}
        A = \sum_{j=1}^r \mathbf{v}_j\mathbf{v}_j^*
    \end{align}
    for some $\mathbf{v_1}$, $\mathbf{v_2}$, $\ldots$, $\mathbf{v_r} \in \mathbb{F}^n$, each with at most $k$ non-zero entries.
\end{definition}

The following basic lemmas will be useful in our study. The first lemma's proof follows immediately from the definition of factor-width-$k$ rank, so we only prove the second one. 

\begin{lemma}\label{lem:add_FW}
    Let $A, B \in \M_n^{+}$. Then  $\mathrm{fran}_k(A+B)\leq \mathrm{fran}_k(A)+\mathrm{fran}_k(B)$.
\end{lemma}

\begin{lemma}\label{lem:DAD}
    Let $D \in \mathcal{M}_n$ be diagonal with non-zero diagonal entries, and let $A \in \mathcal{M}_n^{+}$. Then $A$ has the same factor width and factor-width-$k$ rank as $DAD$.
\end{lemma}

\begin{proof}
    If $A$ has factor width $k$ and $\mathrm{fran}_k(A) = r$ then it can be written in the form
    \[
        A = \sum_{j=1}^r \mathbf{v}_j\mathbf{v}_j^*
    \]
    for some vectors $\mathbf{v_1}$, $\mathbf{v_2}$, $\ldots$, $\mathbf{v_r}$ that each have at most $k$ non-zero entries. Since the vectors $D\mathbf{v_1}$, $D\mathbf{v_2}$, $\ldots$, $D\mathbf{v_r}$ also have at most $k$ non-zero entries, the decomposition
    \[
        DAD^* = \sum_{j=1}^r (D\mathbf{v}_j)(D\mathbf{v}_j)^*
    \]
    shows that $DAD^*$ has factor width at most $k$ and $\mathrm{fran}_k(DAD) \leq r$. Using the invertibility of $D$ establishes the opposite inequalities and thus equality for both factor width and factor-width-$k$ rank.
\end{proof}

Lemma~\ref{lem:DAD} is useful, for example, because it lets us assume without loss of generality that a matrix has diagonal entries equal to $1$ (since we can conjugate by some invertible diagonal $D$ to make this happen, without altering factor width or factor-width-$k$ rank).\footnote{\label{foot:zero_in_diag}Such a diagonal scaling does not exist if one of the diagonal entries of $A \in \M_n^{+}$ equals $0$, but in this case that entire row and column of $A$ must equal $0$, so we can just ignore that row and column and ask about the factor width and factor-width-$k$ rank of the resulting $(n-1) \times (n-1)$ principal submatrix.} This technique was called \emph{diagonal normalization} in \cite{boman2005factor}.

\subsection{Real Versus Complex Factor Width and Factor Width Rank}\label{sec:real_v_complex}

It is natural to ask whether the properties we are investigating depend on the choice of field. In particular, if a matrix $A \in \M_n$ is real, does its factor width or factor width rank depend on whether the ground field is $\mathbb{R}$ or $\mathbb{C}$? In other words, can complex factor width decompositions~\eqref{eq:fac_wid_decomp_rank} of real matrices make use of fewer non-zero entries or fewer terms than real factor width decompositions?

We start by showing that factor width itself does not depend on the ground field, so there is no ambiguity in just talking about ``the'' factor width of a matrix:

\begin{proposition}\label{prop:factor_width_RvsC}
     Suppose $A \in \M_n^+$ has all entries real. The factor width of $A$ does not depend on whether we consider it as a member of $\M_n^{+}(\R)$ or $\M_n^{+}(\C)$.
\end{proposition}

\begin{proof}
    It is clear that the complex factor width of $A$ is less than or equal to its real factor width, so we just need to prove the opposite inequality.

    To this end, suppose $A$ has complex factor width $k$, so we can write
    \begin{align}\label{eq:complex_fack_decomp}
        A = \sum_{j=1}^r \mathbf{v}_j\mathbf{v}_j^*,
    \end{align}
    where each $\mathbf{v}_j \in \C^n$ has at most $k$ non-zero entries. If we let $\mathbf{x}_j := \mathrm{Re}(\mathbf{v}_j)$ and $\mathbf{y}_j := \mathrm{Im}(\mathbf{v}_j)$ be the (entrywise) real and imaginary parts of $\mathbf{v}_j$ then we see that
    \begin{align*}
        A & = \sum_{j=1}^r \mathbf{v}_j\mathbf{v}_j^* = \sum_{j=1}^r (\mathbf{x}_j + i\mathbf{y}_j)(\mathbf{x}_j + i\mathbf{y}_j)^* = \sum_{j=1}^r \mathbf{x}_j\mathbf{x}_j^* + \sum_{j=1}^r \mathbf{y}_j\mathbf{y}_j^* + i\sum_{j=1}^r \big(\mathbf{y}_j\mathbf{x}_j^* - \mathbf{x}_j\mathbf{y}_j^*\big).
    \end{align*}
    Since the entries of $A$ are all real, as are the entries of each $\mathbf{x}_j$ and $\mathbf{y}_j$, the final (imaginary) sum above must equal $0$. This implies
    \begin{align}\label{eq:real_fack_decomp}
        A & = \sum_{j=1}^r \mathbf{x}_j\mathbf{x}_j^* + \sum_{j=1}^r \mathbf{y}_j\mathbf{y}_j^*,
    \end{align}
    which is a decomposition that shows that $A$ has real factor width at most $k$, completing the proof.
\end{proof}

The above proposition demonstrates that we do not need to specify whether the ground field is real or complex when discussing factor width of a real matrix. However, it does not provide us with the same guarantee for factor width rank, since the real factor width decomposition in Equation~\eqref{eq:real_fack_decomp} contains twice as many terms as the complex factor width decomposition in Equation~\eqref{eq:complex_fack_decomp}. We thus obtain the following bounds:

\begin{corollary}\label{cor:frank_inequality_RvC}
    Suppose $A \in \M_n^+$ has all entries real and factor width $k$. Then
    \[
        \mathrm{fran}_k^{\mathbb{C}}(A) \leq \mathrm{fran}_k^{\mathbb{R}}(A) \leq 2\mathrm{fran}_k^{\mathbb{C}}(A).
    \]
\end{corollary}

However, there are also some cases where we can strengthen these inequalities to equality. For example, it is a standard linear algebra fact that the rank of a matrix with real entries does not depend on whether the ground field is taken to be $\R$ or $\C$, so $\mathrm{fran}_n^{\R}(A) = \mathrm{fran}_n^{\C}(A)$ for such matrices $A$. It is similarly not difficult to show that $\mathrm{fran}_1^{\R}(A) = \mathrm{fran}_1^{\C}(A)$. The following proposition shows that an analogous equality holds for factor-width-$2$ rank.

\begin{proposition}\label{prop:real_complex_facwidranktwo}
    Suppose $A \in \M_n^+$ has all entries real and factor width at most $2$. Then
    \[
        \mathrm{fran}_2^{\C}(A) = \mathrm{fran}_2^{\R}(A).
    \]
\end{proposition}

\begin{proof}
    As noted by Corollary~\ref{cor:frank_inequality_RvC}, we have $\mathrm{fran}_k^{\mathbb{C}}(A) \leq \mathrm{fran}_k^{\mathbb{R}}(A)$ for all $k$, so we just need to prove that $\mathrm{fran}_2^{\mathbb{R}}(A) \leq \mathrm{fran}_2^{\mathbb{C}}(A)$.

    To this end, let $r = \mathrm{fran}_2^{\mathbb{C}}(A)$ and let
    \begin{align}\label{eq:complex_facktwo_decomp}
        A = \sum_{\ell=1}^r \mathbf{v}_\ell\mathbf{v}_\ell^*,
    \end{align}
    be a minimal complex factor-width-$2$ decomposition of $A$. Our goal is to construct a real factor-width-$2$ decomposition using no more than $r$ vectors. To this end, first notice that if any $\mathbf{v}_\ell$ contains just one non-zero entry that we can replace it by its entrywise absolute value $\mathbf{w}_\ell = |\mathbf{v}_\ell|$ (which is real) without changing the sum~\eqref{eq:complex_facktwo_decomp}. For this reason, for the remainder of the proof we just consider the vectors $\mathbf{v}_\ell$ that contain two non-zero entries.
    
    Fix integers $1 \leq i \neq j \leq r$ and let $I_{i,j} := \{\ell : [\mathbf{v}_\ell\mathbf{v}_\ell^*]_{i,j} \neq 0 \}$. In other words, $I_{i,j}$ is the set of indices $\ell$ for which the two non-zero entries of $\mathbf{v}_\ell$ are its $i$-th and $j$-th entries. We now split into two cases:
    \begin{itemize}
        \item If $|I_{i,j}| = 1$ then let $m$ denote the single member of $I_{i,j}$. Then the $(i,j)$-entry of $A$ (which is real) equals $[\mathbf{v}_m\mathbf{v}_m^*]_{i,j}$. If we define $\mathbf{w}_m = \mathbf{v}_m/\mathrm{sign}([\mathbf{v}_m]_i)$ (where $\mathrm{sign}([\mathbf{v}_m]_i)$ is the complex sign of the $i$-th entry of $\mathbf{v}_m$) then $\mathbf{w}_m$ is real and we have $\mathbf{w}_m\mathbf{w}_m^* = \mathbf{v}_m\mathbf{v}_m^*$. We can thus replace $\mathbf{v}_m$ by $\mathbf{w}_m$ in the sum~\eqref{eq:complex_facktwo_decomp} without changing its value.

        \item If $|I_{i,j}| \geq 2$ then the matrix
        \[
            \sum_{\ell \in I_{i,j}} \mathbf{v}_\ell\mathbf{v}_\ell^*
        \]
        equals zero outside of a single $2 \times 2$ principal submatrix and thus has rank at most $2$. Since it is real (its only potentially non-zero off-diagonal entries are $a_{i,j}$ and $a_{j,i}$), there exist (via the real singular value decomposition, for example) real vectors $\mathbf{w}_1$ and $\mathbf{w}_2$ with at most $2$ non-zero entries such that
        \[
            \sum_{\ell \in I_{i,j}} \mathbf{v}_\ell\mathbf{v}_\ell^* = \mathbf{w}_1\mathbf{w}_1^* + \mathbf{w}_2\mathbf{w}_2^*.
        \]
    \end{itemize}
    By replacing ``$\mathbf{v}_\ell$'' vectors by real ``$\mathbf{w}_\ell$'' vectors as described above, we can turn the complex factor-width-$2$ decomposition~\eqref{eq:complex_facktwo_decomp} into a real one with no more terms, which completes the proof.
\end{proof}

The proof of the above proposition does not extend straightforwardly to the case when $k \geq 3$, since in these cases the $k \times k$ principal submatrices don't just overlap on the diagonal. However, we are also not aware of any matrix $A \in \M_n(\R)$ and integer $k \geq 3$ for which $\mathrm{fran}_k^{\C}(A) \neq \mathrm{fran}_k^{\R}(A)$. With all this being said, we have an outstanding question which we leave as an open problem.

\begin{question}\label{ques:real_complex_fwr}
    Are the real and complex factor width ranks of a real matrix equal to each other in general? That is, given $A \in \M_n^+(\R)$ with factor width at most $k$, is it true that 
    \[
        \mathrm{fran}_k^{\C}(A) = \mathrm{fran}_k^{\R}(A)?
    \]
\end{question}

In spite of the above question still being open, throughout the remainder of the paper, we prove several bounds on $\mathrm{fran}_k(A)$ that hold regardless of whether the ground field is $\R$ or $\C$, and we therefore state these bounds without specifying the ground field.

\section{Connection with Bandwidth of a Matrix}\label{sec:banded_connection}

Recall that the \emph{bandwidth} of a matrix $A$ is the smallest positive integer $k$ with the property that $a_{i,j} = 0$ whenever $|i-j| \geq k$ (so, for example, diagonal matrices have bandwidth $1$ and tridiagonal matrices have bandwidth $2$) \cite[Chapter 7.2]{bandwidth}.\footnote{Some sources define the bandwidth of a matrix to be $1$ less than the definition used herein, so that diagonal and tridiagonal matrices have bandwidth $0$ and $1$, respectively.} Bandwidth has a natural interpretation as a measure of how ``close'' to diagonal a matrix is: the larger the bandwidth, the farther from being diagonal it is. The same intuition holds for the factor width of a matrix, which perhaps makes the following connection between these two notions unsurprising:

\begin{theorem}\label{thm:factor_width_and_bandwidth}
    If $A \in \mathcal{M}_n^{+}$ has bandwidth $k$ then it has factor width at most $k$. Furthermore, if $k \in \{1,2\}$ then its factor width equals $k$.
\end{theorem}

In particular, diagonal matrices have factor width and bandwidth both equal to $1$ (which we already knew) and tridiagonal matrices with at least $1$ non-zero off-diagonal entry have factor width and bandwidth both equal to $2$.

The equality case of Theorem~\ref{thm:factor_width_and_bandwidth} cannot be extended any further: matrices with factor width $2$ can have any bandwidth. To see this, let $A \in \mathcal{M}_n$ be a matrix width all of its diagonal entries equal to $2k-2$ and all entries on its first $k-1$ super-diagonals and sub-diagonals equal to $1$:
\[
    A = \begin{bmatrix}
        2k-2 & 1 & 1 & \cdots & 1 & 0 & 0 & \cdots & 0 & 0 & 0 \\
        1 & 2k-2 & 1 & \cdots & 1 & 1 & 0 & \cdots & 0 & 0 & 0 \\
        1 & 1 & 2k-2 & \cdots & 1 & 1 & 1 & \cdots & 0 & 0 & 0 \\
        \vdots & \vdots & \vdots & \ddots & \vdots & \vdots & \vdots & \ddots & \vdots & \vdots & \vdots \\
        1 & 1 & 1 & \cdots & 2k-2 & 1 & 1 & \cdots & 1 & 0 & 0 \\
        0 & 1 & 1 & \cdots & 1 & 2k-2 & 1 & \cdots & 1 & 1 & 0 \\
        0 & 0 & 1 & \cdots & 1 & 1 & 2k-2 & \cdots & 1 & 1 & 1 \\
        \vdots & \vdots & \vdots & \ddots & \vdots & \vdots & \vdots & \ddots & \vdots & \vdots & \vdots \\
        0 & 0 & 0 & \cdots & 1 & 1 & 1 & \cdots & 2k-2 & 1 & 1 \\
        0 & 0 & 0 & \cdots & 0 & 1 & 1 & \cdots & 1 & 2k-2 & 1 \\
        0 & 0 & 0 & \cdots & 0 & 0 & 1 & \cdots & 1 & 1 & 2k-2 \\
    \end{bmatrix}.
\]
Then $A$ clearly has bandwidth $k$, but it is diagonally dominant and not diagonal so, by Lemma~\ref{lem:diag_dominance}, has factor width equal to $2$.

In order to prove Theorem~\ref{thm:factor_width_and_bandwidth}, we will first prove a more general result that shows that if the non-zero pattern of a positive semidefinite matrix arises from a chordal graph $G$ then its factor width is bounded by a well-studied graph parameter: the clique number $\omega(G)$ of $G$ (i.e., the size of the largest clique in $G$). To make this statement precise, we first clarify what we mean when we say that the non-zero pattern of a matrix ``arises from'' a graph. Given a (simple, undirected) graph $G = (V,E)$ with $|V| = n$, we consider the following set of positive semidefinite matrices:
\begin{align}\label{eq:nonzero_pattern_G}
    \mathcal{M}_G^{+} \defeq \big\{ A \in \mathcal{M}_n^{+} : a_{i,j} = 0 \ \text{whenever} \ i \neq j \ \text{and} \ (i,j) \notin E \big\}.
\end{align}
In particular, we do not place any restrictions on the diagonal entries of the members of $\mathcal{M}_G^{+}$. Then we have the following result:
    
\begin{theorem}\label{thm:chordal_main}
    Suppose $G$ is a chordal graph (i.e., every cycle of $G$ with four or more vertices has a chord) and $A \in \mathcal{M}_G^{+}$ (where $\mathcal{M}_G^{+}$ is as in Equation~\eqref{eq:nonzero_pattern_G}). Then the factor width of $A$ is no larger than $\omega(G)$. Furthermore, if $\omega(G) \in \{1,2\}$ then the factor width of $A$ is equal to $\omega(G)$. 
\end{theorem}

\begin{proof}
    Since $\mathcal{M}_G^{+}$ is the intersection of the convex set $\mathcal{M}_n^{+}$ and the subspace determined by its zero pattern, $\mathcal{M}_G^{+}$ is convex. It is well-known (see \cite[Theorem~1.1]{AHMR88} or \cite[Theorem~2.4]{PPS89}) that, since $G$ is chordal, the extreme points of $\mathcal{M}_G^{+}$ all have rank $1$.

    It follows that every member of $\mathcal{M}_G^{+}$ can be written as a sum of rank-$1$ members of $\mathcal{M}_G^{+}$. However, each rank-$1$ member $\mathbf{v}\mathbf{v}^* \in \mathcal{M}_G^{+}$ is such that $\mathbf{v}$ has at most $\omega(G)$ non-zero entries (otherwise there would be a principal submatrix of $\mathbf{v}\mathbf{v}^*$ that contains no zeros, and thus a clique of size larger than $\omega(G)$ in $G$). It follows that every member of $\mathcal{M}_G^{+}$ can be written as a sum of matrices of the form $\mathbf{v}\mathbf{v}^*$, where $\mathbf{v}$ has at most $\omega(G)$ non-zero entries (i.e., every member of $\mathcal{M}_G^{+}$ has factor width at most $\omega(G)$).

    The claim about equality when $\omega(G) \in \{1,2\}$ follows from the fact that the sets of matrices with factor width equal to $1$ is equal to the set $\mathcal{M}_G^{+}$ when $G$ is the graph with $n$ unconnected vertices (i.e., the graph on $n$ vertices with $\omega(G) = 1$): both of these sets equal the set of diagonal positive semidefinite matrices.
\end{proof}

\begin{proof}[Proof of Theorem~\ref{thm:factor_width_and_bandwidth}]
    Observe that if $G = (V,E)$ with $V = \{1,2,\ldots,n\}$ and $E = \{(i,j) : |i-j| < k \}$ then $G$ is chordal, $\omega(G) = k$, and $\mathcal{M}_G^{+}$ is the set of positive semidefinite matrices with bandwidth at most $k$. Applying Theorem~\ref{thm:chordal_main} gives the result.\footnote{We give an alternate proof of this fact, which does not use Theorem~\ref{thm:chordal_main} or any graph theory, a bit later in the proof of Theorem~\ref{thm:tridiag_fac_two_rank}.}
\end{proof}

In the $k = 2$ case, we can make Theorem~\ref{thm:factor_width_and_bandwidth} slightly more explicit by demonstrating how to come up with a factor-width-$2$ decomposition of a tridiagonal matrix:

\begin{example}\label{exam:bandwidth_two} Let $A\in \mathcal M_n^+$ be tridiagonal. 
    Let $a_1, a_2, \ldots, a_n$ denote the diagonal entries of $A$, while $b_1, b_2, \ldots, b_{n-1}$ denote the entries on the first (upper) superdiagonal of $A$. By Sylvester's criterion, we know that if $A$ is positive semidefinite then all of its principal submatrices have non-negative determinant. If $d_k$ denotes the determinant of the top-left $k \times k$ principal submatrix of $A$, then performing a cofactor expansion across the $k$-th row of $A$ tells us that
    \begin{align}\label{eq:tridiag_det_rec}
        d_k = \begin{cases}
            a_k, & \ \text{if $k = 1$,} \\
            a_kd_{k-1} - |b_{k-1}|^2d_{k-2}, &  \ \text{otherwise}.
        \end{cases}
    \end{align}
    
    If $j$ is the smallest positive integer for which $d_j = 0$ then the fact that $d_{j+1} \geq 0$ tells us, via the recurrence~\eqref{eq:tridiag_det_rec}, that $b_j = 0$. This implies that $A$ is block diagonal, where each block is tridiagonal with the property that the determinant of the submatrix obtained by erasing its final row and column is non-zero. For the remainder of the proof, we assume that $A$ itself has this property, but if it is not then the same argument can be applied to each of its diagonal blocks.
    
    By assumption, we have $d_k > 0$ for all $1 \leq k < n$. If we define $s_1 := d_1$ and $s_k := d_k/d_{k-1}$ for $2 \leq k \leq n$ then the recurrence~\eqref{eq:tridiag_det_rec} tells us that $s_k = a_k - |b_{k-1}|^2/s_{k-1}$ for all $2 \leq k \leq n$. Since $d_k > 0$ for all $2 \leq k < n$, we have $s_k > 0$ as well (and $s_n \geq 0$). It follows that we can write
    \begin{align}\begin{split}\label{eq:tridiag_decomp}
        A & = \begin{bmatrix}
            s_1 & b_1 & 0 & \cdots & 0 & 0 \\
            \overline{b_1} & |b_1|^2/s_1 & 0 & \cdots & 0 & 0 \\
            0 & 0 & 0 & \cdots & 0 & 0 \\
            \vdots & \vdots & \vdots & \ddots & \vdots & \vdots \\
            0 & 0 & 0 & \cdots & 0 & 0 \\
            0 & 0 & 0 & \cdots & 0 & 0
        \end{bmatrix} + \begin{bmatrix}
            0 & 0 & 0 & \cdots & 0 & 0 \\
            0 & s_2 & b_2 & \cdots & 0 & 0 \\
            0 & \overline{b_2} & |b_2|^2/s_2 & \cdots & 0 & 0 \\
            \vdots & \vdots & \vdots & \ddots & \vdots & \vdots \\
            0 & 0 & 0 & \cdots & 0 & 0 \\
            0 & 0 & 0 & \cdots & 0 & 0
        \end{bmatrix} + \cdots \\
        & \quad \cdots + \begin{bmatrix}
            0 & 0 & 0 & \cdots & 0 & 0 \\
            0 & 0 & 0 & \cdots & 0 & 0 \\
            0 & 0 & 0 & \cdots & 0 & 0 \\
            \vdots & \vdots & \vdots & \ddots & \vdots & \vdots \\
            0 & 0 & 0 & \cdots & s_{n-1} & b_{n-1} \\
            0 & 0 & 0 & \cdots & \overline{b_{n-1}} & |b_{n-1}|^2/s_{n-1}
        \end{bmatrix} + \begin{bmatrix}
            0 & 0 & 0 & \cdots & 0 & 0 \\
            0 & 0 & 0 & \cdots & 0 & 0 \\
            0 & 0 & 0 & \cdots & 0 & 0 \\
            \vdots & \vdots & \vdots & \ddots & \vdots & \vdots \\
            0 & 0 & 0 & \cdots & 0 & 0 \\
            0 & 0 & 0 & \cdots & 0 & s_n
        \end{bmatrix},
    \end{split}\end{align}
    which is an explicit decomposition showing that $A$ has factor width at most $2$.
\end{example}

Theorem~\ref{thm:factor_width_and_bandwidth} tells us that every matrix with bandwidth $k$ has a well-defined (finite) factor-width-$k$ rank. Our next result shows that this factor-width-$k$ rank just equals the usual rank of the matrix:

\begin{theorem}\label{thm:tridiag_fac_two_rank}
    If $A \in \mathcal{M}_n^{+}$ has bandwidth at most $k$ then $\mathrm{fran}_k(A) = \mathrm{rank}(A)$.
\end{theorem}

\begin{proof}
    The inequality $\mathrm{fran}_k(A) \geq \mathrm{rank}(A)$ follows immediately from Definition~\ref{defn:factor_width}: one well-known characterization of $\mathrm{rank}(A)$ is as the minimum number of terms in a decomposition of the form~\eqref{eq:fac_wid_decomp} where no restrictions are placed on the $\mathbf{v}_j$ vectors. We thus just need to prove the inequality $\mathrm{fran}_k(A) \leq \mathrm{rank}(A)$.
    
    To this end, define $r := \mathrm{rank}(A)$. Recall that every positive semidefinite matrix $A \in \mathcal{M}_n^{+}$ has a unique Cholesky decomposition $A = LL^*$ with $L$ lower triangular, having $r$ strictly positive diagonal entries, and $n-r$ columns that are equal to $\mathbf{0}$ \cite[Section~3.2.2]{Gen98}. Furthermore, this lower triangular matrix $L$ has bandwidth no larger than that of $A$ \cite[Algorithm~3.1]{Gen98}. Since $L$ has bandwidth at most $k$ and is lower triangular, its columns $\mathbf{v}_1$, $\mathbf{v}_2$, $\ldots$, $\mathbf{v}_n$ each have at most $k$ non-zero entries. Since at most $r$ of these columns are non-zero (which we will assume WLOG are $\mathbf{v}_1$, $\mathbf{v}_2$, $\ldots$, $\mathbf{v}_r$), we have
    \[
        A = LL^* = \sum_{j=1}^r \mathbf{v}_j\mathbf{v}_j^*,
    \]
    which is a decomposition of $A$ that shows that it has factor-width-$k$ rank at most $r$. This completes the proof.

\end{proof}

\section{Factor Width 2 Rank}\label{sec:factor_width_2_rank}

We now investigate some bounds on the factor-width-2 rank of a matrix, and some special cases in which we can compute it explicitly. We will first need some lemmas that help us specifically in the factor width $2$ case. Recall that a matrix $A\in \mathcal M_n$ is \emph{diagonally dominant} if $|a_{i,i}|\geq \sum_{i\neq j} |a_{i,j}|$ for all $1\leq i\leq n$. We restate \cite[Proposition~2]{boman2005factor} as a lemma here.

\begin{lemma}\label{lem:diag_dominance}
    If $A \in \mathcal{M}_n^{+}$ is diagonally dominant then it has factor width at most $2$.
\end{lemma}
 
A matrix is \emph{scaled diagonally dominant}, also called \emph{generalized diagonally dominant}, if there exists a diagonal matrix $D \in \M_n^{+}$ for which $DAD$ is diagonally dominant \cite{bishan}. Our interest in scaled diagonal dominance comes from the fact that it is equivalent to factor width $2$. In particular, a careful reading of \cite[Theorems 8 and 9]{boman2005factor} yields the following result:

\begin{lemma}\label{lem:SDD}
    A matrix $A \in \mathcal{M}_n^{+}$ has factor width at most $2$ if and only if $A$ is scaled diagonal dominant.
\end{lemma}

With the above known results out of the way, we now proceed to deriving bounds on the factor-width-$2$ rank of a matrix. Many of these bounds will depend on the number of non-zero entries strictly above the matrix's main diagonal. For a given matrix $A$, we denote this quantity by $\mathrm{nnzu}(A)$, and we denote the total number of non-zero off-diagonal entries of $A$ by $\mathrm{nnz}(A)$. For example, if all entries of $A \in \M_n^{+}$ are non-zero then $\mathrm{nnzu}(A) = n(n-1)/2$, and in general we have  $\mathrm{nnzu}(A) = \mathrm{nnz}(A)/2$.

\begin{lemma}\label{lem:fac_wid_two_nonzero_lb}
    Suppose $A \in \mathcal{M}_n^{+}$ has factor width $2$. Then $\mathrm{fran}_2(A) \geq \mathrm{nnzu}(A)$.
\end{lemma}

The above lemma is straightforward to prove, but we omit the proof since we prove a generalization of it a bit later, as Proposition~\ref{prop:fac_wid_two_nonzero}. Remarkably, under some additional hypotheses we actually get equality in the above lemma:

\begin{theorem}\label{thm:fac_wid_rank2_3x3}
    Suppose $n \geq 3$, $A \in \mathcal{M}_n^{+}$ has factor width~$2$, and each diagonal entry of $A$ belongs to some $3 \times 3$ principal submatrix with all entries non-zero. Then
    \[
        \mathrm{fran}_2(A) = \mathrm{nnzu}(A).
    \]
\end{theorem}

\begin{proof}
    By Lemma~\ref{lem:SDD}, we can assume without loss of generality that $A$ is diagonally dominant. We start by proving the proposition in the special case each of the diagonal dominance inequalities hold with equality; that is, when
    \begin{align}\label{eq:dd_equality}
        a_{j,j} = \sum_{i \neq j}|a_{i,j}| \quad \text{for all} \quad 1 \leq j \leq n.
    \end{align}
    
    In this special case, we can write
    \begin{align}\label{eq:fac_wid_2_explicit_rank}
        A = \sum_{i=1}^n\sum_{j=i+1}^n \mathbf{v}_{i,j}\mathbf{v}_{i,j}^*,
    \end{align}
    where $\mathbf{v}_{i,j} = \sqrt{|a_{i,j}|}(\mathrm{sgn}(a_{i,j})\mathbf{e}_i + \mathbf{e}_j)$ and $\mathrm{sgn}(a_{i,j}) = a_{i,j}/|a_{i,j}|$ is the complex sign of $a_{i,j}$. Since each $\mathbf{v}_{i,j}$ has at most $2$ (in fact, exactly $2$) non-zero entries, and the sum~\eqref{eq:fac_wid_2_explicit_rank} contains $n(n-1)/2$ terms, $\mathrm{nnz}(A)/2$ of which are non-zero, which shows that $\mathrm{fran}_2(A) \leq \mathrm{nnz}(A)/2$. When combined with Lemma~\ref{lem:fac_wid_two_nonzero_lb}, this tells us that $\mathrm{fran}_2(A) = \mathrm{nnz}(A)/2$, completing the proof of the special case where Equation~\eqref{eq:dd_equality} holds.

    To complete the proof of the theorem in general, we will prove the following claim:

    \textbf{Claim~(a):} If $A$ has factor width $2$, $\mathrm{fran}_2(A) = \mathrm{nnz}(A)/2$, and each diagonal entry of $A$ belongs to some $3 \times 3$ principal submatrix with all entries non-zero, then $\mathrm{fran}_2(A + d\mathbf{e}_m\mathbf{e}_m^*) = \mathrm{fran}_2(A)$ for all indices $m$ and all scalars $d > 0$.

    The theorem follows from Claim~(a) since every diagonally dominant matrix $A$ can be written in the form $A = B + \sum_j d_j \mathbf{e}_j\mathbf{e}_j^*$ for some non-negative scalars $\{d_j\}$, where $B$ is diagonally dominant with equality (i.e., satisfying Equation~\eqref{eq:dd_equality}), so
    \[
        \mathrm{fran}_2(A) = \mathrm{fran}_2(B) = \mathrm{nnz}(B)/2 = \mathrm{nnz}(A)/2.
    \]
    
    To prove claim~(a), fix an index $m$ and a real scalar $d > 0$, and suppose $A$ has factor-width-2 decomposition
    \begin{align}\label{eq:fac_wid_2_nnz_induct}
        A = \sum_{\ell=1}^{\mathrm{nnz}(A)/2} \mathbf{v}_\ell\mathbf{v}_\ell^*.
    \end{align}
    Since this sum contains $\mathrm{nnz}(A)/2$ terms, each non-zero off-diagonal entry $a_{i,j}$ of $A$ is equal to $[\mathbf{v}_\ell\mathbf{v}_\ell^*]_{i,j}$ for some $\ell$. It follows that if we choose a $3 \times 3$ principal submatrix of $A$ containing its $(m,m)$-entry and with all entries non-zero, there are exactly $3$ terms from the sum~\eqref{eq:fac_wid_2_nnz_induct} that contribute to the off-diagonal entries of this submatrix. Call those vectors $\mathbf{x} = (r,s,0)$, $\mathbf{y} = (u,0,t)$, and $\mathbf{z} = (0,v,w)$, and write
    \begin{align}\label{eq:3x3_adjust}
        \mathbf{x}\mathbf{x}^* + \mathbf{y}\mathbf{y}^* + \mathbf{z}\mathbf{z}^* = \begin{bmatrix}
            |r|^2 + |u|^2 & r\overline{s} & u\overline{t} \\
            \overline{r}s & |s|^2 + |v|^2 & v\overline{w} \\
            \overline{u}t & \overline{v}w & |t|^2 + |w|^2
        \end{bmatrix},
    \end{align}
    where we have only displayed the $3$ rows and columns of interest.\footnote{The off-diagonal entries are also off-diagonal entries of $A$ itself. However, the diagonal entries need not be equal to the corresponding entries of $A$, since other terms in the sum~\eqref{eq:fac_wid_2_nnz_induct} can contribute to those entries.} In other words, the $3 \times 3$ principal submatrix~\eqref{eq:3x3_adjust} has factor-width-2 rank equal to $3$, and our goal is to show that if we add a positive scalar to one of its diagonal entries then it still has factor-width-2 rank equal to $3$. We state this as a new claim:

    \textbf{Claim~(b):} If $A \in \M_3^{+}$ has factor width $2$, $\mathrm{fran}_2(A) = 3$, and each off-diagonal entry of $A$ is non-zero, then $\mathrm{fran}_2(A + d\mathbf{e}_m\mathbf{e}_m^*) = 3$ for all indices $m$ and all scalars $d > 0$.

    To prove claim~(b) (and thus the theorem), we note that Lemma~\ref{lem:DAD} lets us assume without loss of generality that every diagonal entry of $A$ is equal to $1$. Then we can write
    \[
        A = \begin{bmatrix}
            1 & a & \overline{b} \\
            \overline{a} & 1 & c \\
            b & \overline{c} & 1
        \end{bmatrix}
    \]
    for some scalars $a,b,c \in \mathbb{F}$. Since $\mathrm{fran}_2(A) = 3$ and $a,b,c \neq 0$, we can write
    \[
        A = \mathbf{u}\mathbf{u}^* + \mathbf{v}\mathbf{v}^* + \mathbf{w}\mathbf{w}^*,
    \]
    for some $\mathbf{u} = (a/y, y, 0)$, $\mathbf{v} = (x, 0, b/x)$, $\mathbf{w} = (0, c/z, z)$, and $x,y,z \in \mathbb{F}$.

    Now let $\xi > x$ and define the vectors $\mathbf{u}_\xi = (a/y_\xi, y_\xi, 0)$, $\mathbf{v}_\xi = (\xi, 0, b/\xi)$, $\mathbf{w}_\xi = (0, c/z_\xi, z_\xi)$, where
    \begin{align*}
        z_\xi & = z^2 + |b|^2\left(\frac{1}{x^2} - \frac{1}{\xi^2}\right) \quad \text{and} \\
        y_\xi & = y^2 + |c|^2\left(\frac{1}{z^2} - \frac{1}{z^2 + |b|^2\left(\frac{1}{x^2} - \frac{1}{\xi^2}\right)}\right).
    \end{align*}
    Note that $\lim_{\xi \rightarrow x^{+}}\mathbf{u}_\xi = \mathbf{u}$, $\lim_{\xi \rightarrow x^{+}}\mathbf{v}_\xi = \mathbf{v}$, and $\lim_{\xi \rightarrow x^{+}}\mathbf{w}_\xi = \mathbf{w}$. A straightforward (but hideous) calculation then reveals that
    \[
        \mathbf{u}_\xi\mathbf{u}_\xi^* + \mathbf{v}_\xi\mathbf{v}_\xi^* + \mathbf{w}_\xi\mathbf{w}_\xi^* = \begin{bmatrix}
            f(\xi) & a & \overline{b} \\
            \overline{a} & 1 & c \\
            b & \overline{c} & 1
        \end{bmatrix},
    \]
    where
    \begin{align}\label{eq:xi_function}
        f(\xi) = \xi^2 + \frac{|a|^2z^2((|b|^2 + x^2z^2)\xi^2 - |b|^2x^2)}{ x^2y^2z^4\xi^2 + |b|^2(\xi^2 - x^2)(|c|^2 + y^2z^2) }.
    \end{align}
    Then $f$ is a rational function of $\xi$ that is continuous on $(x,\infty)$ and has $\lim_{\xi \rightarrow x^{+}} f(\xi) = x^2 + |a|^2/y^2 = 1$. Furthermore, the fraction on the right-hand-side of Equation~\eqref{eq:xi_function} has a finite limit as $\xi \rightarrow \infty$, so $\lim_{\xi \rightarrow \infty} f(\xi) = \infty$. It follows that the range of $f$ contains $(1,\infty)$, so every matrix of the form
    \[
        \begin{bmatrix}
            d & a & \overline{b} \\
            \overline{a} & 1 & c \\
            b & \overline{c} & 1
        \end{bmatrix}
    \]
    for $d > 1$ has factor-width-2 rank equal to $3$. This completes the proof of claim~(b) and of the theorem.
\end{proof}

We note that the strange hypothesis in the statement of Theorem~\ref{thm:fac_wid_rank2_3x3} involving $3 \times 3$ cannot be dropped. For example, tridiagonal matrices do not satisfy that hypothesis, and we saw in Theorems~\ref{thm:factor_width_and_bandwidth} and~\ref{thm:tridiag_fac_two_rank} that if $A \in \M_n^{+}$ is tridiagonal with $\mathrm{nnzu}(A) \geq 1$ then $\mathrm{fran}_2(A) = \mathrm{rank}(A)$, which can be strictly larger than $\mathrm{nnzu}(A)$.

In the special case when $A \in \mathcal{M}_n^{+}$ has all of its entries non-zero then we have $\mathrm{nnzu}(A) = n(n-1)/2$, leading to the following corollary of Theorem~\ref{thm:fac_wid_rank2_3x3}:

\begin{corollary}\label{cor:fac_wid_rank2_nonzero}
    If $n \geq 3$ and $A \in \mathcal{M}_n^{+}$ has factor width~$2$ and no entries equal to $0$ then
    \[
        \mathrm{fran}_2(A) = \frac{n(n-1)}{2}.
    \]
\end{corollary}

Again, the hypothesis that $n \geq 3$ in the above corollary cannot be dropped, as demonstrated by any rank-$2$ matrix $A \in \mathcal{M}_2^{+}$ with all entries non-zero.

There are a few other cases where we can easily compute the factor-width-2 rank of a matrix. Recall that an \emph{arrowhead matrix} $A \in \M_n$ is one for which $a_{i,j} = 0$ whenever $i \neq 1$, $j \neq 1$, and $i \neq j$. In other words, it is a matrix of the form
\[
    A = \begin{bmatrix}
        * & * & * & * & \cdots & * \\
        * & * & 0 & 0 & \cdots & 0 \\
        * & 0 & * & 0 & \cdots & 0 \\
        * & 0 & 0 & * & \cdots & 0 \\
        \vdots & \vdots & \vdots & \vdots & \ddots & \vdots \\
        * & 0 & 0 & 0 & \cdots & *
    \end{bmatrix}.
\]
Our next result completely determines the factor width and factor width rank of arrowhead matrices:

\begin{theorem}\label{thm:arrowhead_facwid2}
    Suppose $A \in \M_n^+$ is an arrowhead matrix. Then $A$ has factor width at most $2$. Furthermore, $\mathrm{fran}_k(A) = \rank(A)$ for all $k \geq 2$.
\end{theorem}

\begin{proof}
    One way to establish the fact that $A$ has factor width at most $2$ is to apply Theorem~\ref{thm:chordal_main}: if $G$ is such that $A \in \mathcal{M}_G^{+}$ then $G$ is a tree and is thus chordal and has $\omega(G) = 2$. However, since we also want to compute the factor width rank of $A$, we will also compute an explicit factor-width-2 decomposition of $A$.

    Assume without loss of generality that all of the diagonal entries of $A$ are non-zero. If any of them \emph{do} equal $0$, we can delete that entire row and column since doing so does not alter $\mathrm{fran}_k(A)$ or $\rank(A)$. Also, it suffices to just prove that $\mathrm{fran}_2(A) \leq \rank(A)$, since the inequalities
    \[
        \rank(A) = \mathrm{fran}_n(A) \leq \mathrm{fran}_{n-1}(A) \leq \cdots \leq \mathrm{fran}_3(A) \leq \mathrm{fran}_2(A)
    \]
    always hold.

    Furthermore, since we are assuming that the diagonal entries of $A$ are all non-zero, $\rank(A) \geq n-1$ (the rank of $A$'s bottom-right $(n-1) \times (n-1)$ principal submatrix), so either $\rank(A) = n$ (if $\det(A) > 0$) or $\rank(A) = n-1$ (if $\det(A) = 0$). It is straightforward to use the sum-over-permutations formula of the determinant to compute
    \[
        \det(A) = \left(\prod_{i=2}^na_{i,i}\right)\left( a_{1,1} - \sum_{i=2}^n \frac{|a_{1,i}|^2}{a_{i,i}}\right).
    \]

    In particular, this implies (since $A$ is positive semidefinite so $\det(A) \geq 0$) that $a_{1,1} \geq \sum_{i=2}^n |a_{1,i}|^2/a_{i,i}$. For each $2 \leq i \leq n$, define $\mathbf{v}_i$ to be the vector with only two non-zero entries, which are $[\mathbf{v}_i]_i = \sqrt{a_{i,i}}$ and $[\mathbf{v}_i]_1 = a_{1,i}/\sqrt{a_{i,i}}$. It is then straightforward to verify that
    \[
        [\mathbf{v}_i\mathbf{v}_i^*]_{1,1} = |a_{1,i}|^2/a_{i,i}, \quad [\mathbf{v}_i\mathbf{v}_i^*]_{1,i} = a_{1,i}, \quad \text{and} \quad [\mathbf{v}_i\mathbf{v}_i^*]_{i,i} = a_{i,i}.
    \]
    It follows that
    \[
        A = \sum_{i=2}^n \mathbf{v}_i\mathbf{v}_i^* + \begin{bmatrix}
            a_{1,1} - \sum_{i=2}^n \frac{|a_{1,i}|^2}{a_{i,i}} & 0 & 0 & \cdots & 0 \\
            0 & 0 & 0 & \cdots & 0 \\
            0 & 0 & 0 & \cdots & 0 \\
            \vdots & \vdots & \vdots & \ddots & \vdots \\
            0 & 0 & 0 & \cdots & 0
        \end{bmatrix},
    \]
    which is a factor-width-2 decomposition of $A$ containing $n-1$ terms if $\det(A) = 0$ and containing $n$ terms if $\det(A) > 0$.
\end{proof}

Another special type of matrix for which we can easily compute the factor-width-2-rank is when the matrix $A$ is block diagonal with overlapping $1 \times 1$ block corners, and no entries equal to $0$ in those blocks. That is, when the matrix can be written in the form
\[
    \begin{bmatrix}
        A_{1,1} & \mathbf{b_1} & O &  \mathbf{0} & \cdots \\
        \mathbf{b_1}^* & c_1 & \mathbf{d_1}^* & 0 & \cdots \\
        O & \mathbf{d_1} & A_{2,2} & \mathbf{b_2} & \cdots \\
        \mathbf{0}^* & 0 & \mathbf{b_2}^* & c_2 & \cdots \\
        \vdots & \vdots & \vdots & \vdots & \ddots \\
    \end{bmatrix},
\]
where the $c_j$'s are all $1 \times 1$ (i.e., scalars).

\begin{theorem}\label{thm:block_diag}
     Let $n \geq 3$ and $A \in \mathcal{M}_n^{+}$ be a block diagonal matrix each of whose blocks overlap the next block by exactly one entry, and each having no entries equal to $0$. Further, suppose $A$ has factor width~$2$. Then
    \[
        \mathrm{fran}_2(A) = \mathrm{nnzu}(A).
    \]
\end{theorem}

\begin{proof}
    We consider just the case where there are two overlapping diagonal blocks; the general result can then be readily established via induction.

    Suppose $A$ has the form
    \[
        A = \begin{bmatrix}
            A_{1,1} & \mathbf{b} & O \\
            \mathbf{b}^* & c & \mathbf{d}^* \\
            O & \mathbf{d} & A_{2,2}
        \end{bmatrix},
    \]
    where $A_{1,1} \in \M_s^{+}$ and $A_{2,2} \in \M_t^{+}$ are square matrices (with $n = s+t+1$), $\mathbf{b} \in \C^s$ and $\mathbf{d} \in \C^t$ are column vectors, and $c$ is a scalar. Since $\mathrm{nnzu}(A) = s(s+1)/2 + t(t+1)/2$, we know from Lemma~\ref{lem:fac_wid_two_nonzero_lb} that $\mathrm{fran}_2(A) \geq s(s+1)/2 + t(t+1)/2$. All that remains is to prove the opposite inequality.

    Define
    \[
        \widetilde{g} := \min_{g \in \R}\left\{\begin{bmatrix}
            A_{1,1} & \mathbf{b} \\
            \mathbf{b}^* & g
        \end{bmatrix} \text{has factor width at most $2$}\right\}
    \]
    (note that such a $g$ exists since we could choose $g = c$, for example, and the minimum is attained thanks to the fact that the set of matrices with factor width at most $2$ is closed). Then the matrices
    \[
        \begin{bmatrix}
            A_{1,1} & \mathbf{b} \\
            \mathbf{b}^* & \widetilde{g}
        \end{bmatrix} \quad \text{and} \quad \begin{bmatrix}
            c-\widetilde{g} & \mathbf{d}^* \\
            \mathbf{d} & A_{2,2}
        \end{bmatrix}
    \]
    both have factor width equal to $2$ and all off-diagonal entries non-zero. It follows from Corollary~\ref{cor:fac_wid_rank2_nonzero} that 
    \[
        \mathrm{fran}_2\left(\begin{bmatrix}
            A_{1,1} & \mathbf{b} \\
            \mathbf{b}^* & \widetilde{g}
        \end{bmatrix}\right) = \frac{s(s+1)}{2} \quad \text{and} \quad \mathrm{fran}_2\left(\begin{bmatrix}
            c-\widetilde{g} & \mathbf{d}^* \\
            \mathbf{d} & A_{2,2}
        \end{bmatrix}\right) = \frac{t(t+1)}{2}.
    \]
    By Lemma~\ref{lem:add_FW}, it follows that
    \[
        \mathrm{fran}_2(A) \leq \mathrm{fran}_2\left(\begin{bmatrix}
            A_{1,1} & \mathbf{b} & O \\
            \mathbf{b}^* & \widetilde{g} & \mathbf{0}^* \\
            O & \mathbf{0} & O
        \end{bmatrix}\right) + \mathrm{fran}_2\left(\begin{bmatrix}
            O & \mathbf{0} & O \\
            \mathbf{0}^* & c-\widetilde{g} & \mathbf{d}^* \\
            O & \mathbf{d} & A_{2,2}
        \end{bmatrix}\right) = \frac{s(s+1)}{2} + \frac{t(t+1)}{2},
    \]
    as claimed.
\end{proof}

\section{General Bounds on Factor Width Rank}\label{sec:factor_width_rank_bounds}

Our primary concern in this section is how to bound and/or compute the factor-width-$k$ rank of a matrix. While exact computation of this quantity in general seems difficult (factor width itself can be determined via semidefinite programming, but it is not clear whether or not the same is true of factor-width-$k$ rank), there are numerous bounds that can at least narrow it down or determine its value in some special cases. We start with the most basic ones:

\begin{proposition}\label{prop:two_rank_ubs}
    Suppose $A \in \mathcal{M}_n^{+}$ has factor width less than or equal to $k$. Then
    \begin{align*}
        \mathrm{fran}_k(A) \leq k\binom{n}{k}.
    \end{align*}
\end{proposition}

\begin{proof}
    If $A$ has factor width at most $k$ then it can be written as a sum of $\binom{n}{k}$ positive semidefinite matrices, each of which is equal to zero outside of one of the $\binom{n}{k}$ possible $k \times k$ principal submatrices. By writing each of these matrices as a sum of $k$ rank-$1$ positive semidefinite matrices, we obtain the desired bound.
\end{proof}

For intermediate values of $k$ (e.g., when $k$ is near $n/2$), much better upper bounds are possible:

\begin{proposition}\label{prop:fac_wid_rank_basic_bounds}
    Suppose $A \in \mathcal{M}_n^{+}(\mathbb{F})$ ($\mathbb{F} = \R$ or $\mathbb{F} = \C$) has factor width less than or equal to $k$. Then
    \begin{align}\label{eq:caratheodory_frank_bound}
        \mathrm{rank}(A) \leq \mathrm{fran}_k^{\mathbb{F}}(A) \leq \begin{cases}
            n(n+1)/2 & \text{if \ $\mathbb{F} = \R$}, \\
            n^2 & \text{if \ $\mathbb{F} = \C$}.
        \end{cases}
    \end{align}
\end{proposition}

\begin{proof}
    As already mentioned earlier, the inequality $\mathrm{rank}(A) \leq \mathrm{fran}_k(A)$ follows straight from the relevant definitions.

    The other inequality comes from Carath\'{e}odory's theorem \cite[Theorem~1.9]{Wat18}, which says that every member of a convex set $\mathcal{C}$ living inside of a $d$-dimensional real affine space $\mathcal{V}$ can be written as a convex combination of $d+1$ of the set's extreme points. For our purposes, let's consider the convex set and real affine space
    \begin{align*}
        \mathcal{C} & := \big\{ A \in \mathcal{M}_n : \tr(A) = 1 \ \text{for all $j$, $A$ has factor width at most $k$} \big\}, \\
        \mathcal{V} & := \big\{ A \in \mathcal{M}_n : \tr(A) = 1 \ \text{for all $j$,} \ A^* = A \big\},
    \end{align*}
    respectively. If $\mathbb{F} = \R$ then $\mathrm{dim}(\mathcal{V}) = n(n+1)/2-1$, and if $\mathbb{F} = \C$ then $\mathrm{dim}(\mathcal{V}) = n^2-1$ (we emphasize that, even if $\mathbb{F} = \C$, $\mathcal{V}$ is indeed a real affine space with \emph{real} dimension $n^2-1$). Carath\'{e}odory's theorem thus tells us that every factor-width-$k$ matrix $A$ with trace equal to $1$ has $\mathrm{fran}_k(A) \leq \mathrm{dim}(\mathcal{V}) + 1$, thus attaining the bound~\eqref{eq:caratheodory_frank_bound} in this case. To see that the same is true even if $\tr(A) \neq 1$ we can simply multiply $A$ by an appropriate scalar.
\end{proof}

Since $\rank(A) \leq n$ for all $A \in \mathcal{M}_n^{+}$, there is quite a gap between the bounds provided by Proposition~\ref{prop:fac_wid_rank_basic_bounds}. Nevertheless, the results presented in the remainder of this section will show that neither bound can be improved significantly in general.

For example, the inequality $\mathrm{rank}(A) \leq \mathrm{fran}_k(A)$ is actually an equality whenever $k = 1$: for a diagonal positive semidefinite matrix $A$, its rank and factor-width-$1$ rank are both equal to its number of non-zero diagonal entries. Similarly, we saw in Theorem~\ref{thm:tridiag_fac_two_rank} that $\mathrm{fran}_k(A) = \mathrm{rank}(A)$ whenever $A$ has bandwidth less than or equal to $k$.

At the other extreme, when $k \geq 2$ there are also cases where $\mathrm{fran}_k(A)$ is close to the upper bound provided by Proposition~\ref{prop:fac_wid_rank_basic_bounds} (we already stated this fact for $k = 2$ in Lemma~\ref{lem:fac_wid_two_nonzero_lb}, but we now prove it for all $k$):

\begin{proposition}\label{prop:fac_wid_two_nonzero}
    Suppose $A \in \mathcal{M}_n^{+}$ has factor width $k \geq 2$. Then
    \[
        \mathrm{fran}_k(A) \geq \left(\frac{2}{k(k-1)}\right)\mathrm{nnzu}(A).
    \]
\end{proposition}

\begin{proof}
    In any decomposition of the form
    \[
        A = \sum_{j=1}^r \mathbf{v}_j\mathbf{v}_j^*,
    \]
    where each $\mathbf{v}_j$ has at most $k$ non-zero entries, each term $\mathbf{v}_j\mathbf{v}_j^*$ has at most $k(k-1)/2$ non-zero entries above its main diagonal. It follows that $\mathrm{nnzu}(A) \leq rk(k-1)/2$. Rearranging gives the desired inequality.
\end{proof}

If $A \in \mathcal{M}_n^{+}$ has factor width $k$ and all of its entries non-zero then the above proposition shows that $\mathrm{fran}_k(A) \geq \frac{n(n-1)}{k(k-1)}$. In particular, for any fixed $k \geq 2$, this shows that the upper bound of Proposition~\ref{prop:fac_wid_rank_basic_bounds} is asymptotically tight: no sub-quadratic (in $n$) upper bound on $\mathrm{fran}_k(A)$ is possible. If we restrict our attention further to the $k = 2$ case, then this tells us that every factor-width-$2$ matrix $A \in \mathcal{M}_n^{+}$ with no entries equal to $0$ has
\begin{align}\label{ineq:fran2}
    \mathrm{fran}_2(A) \geq \frac{n(n-1)}{2}.
\end{align}
If $n = 2$ then this inequality might be strict, since there clearly exist $2 \times 2$ matrices with $\mathrm{fran}_2(A) = \mathrm{rank}(A) = 2 > 1$. On the other hand, in Corollary~\ref{cor:fac_wid_rank2_nonzero} we showed that if $n \geq 3$ then Inequality~\eqref{ineq:fran2} is actually an equality.

\subsection{Better Lower Bounds via Coverings}\label{sec:covering_design_bounds}

When all entries of a matrix are non-zero and $k = 2$, Corollary~\ref{cor:fac_wid_rank2_nonzero} shows that the inequality $\mathrm{fran}_k(A) \geq \frac{n(n-1)}{k(k-1)}$ of Proposition~\ref{prop:fac_wid_two_nonzero} holds with equality. However, when $k \geq 3$ this is no longer the case: stronger lower bounds are possible. To develop these stronger lower bounds, we recall a combinatorial object that will be of use to us \cite{ER56,EH63}:

\begin{definition}\label{defn:covering_design}
    Let $t \leq k \leq n$ be positive integers. An \emph{$(n,k,t)$-covering design} is a collection of $k$-element subsets of $[n]$ with the property that any $t$-element subset of $[n]$ is contained in at least one of the $k$-element subsets. The minimum number of $k$-element subsets possible in an $(n,k,t)$-covering design is denoted by $C(n,k,t)$.
\end{definition}

Our interest in covering designs comes from the following proposition, which shows that they bound the factor-width-$k$ rank of a matrix whose entries are all non-zero:

\begin{proposition}\label{prop:fac_wid_covering_design}
    Suppose $A \in \mathcal{M}_n^{+}$ has factor width $k \geq 2$ and all of its entries non-zero. Then
    \begin{align}\label{eq:frank_bigger_cnk}
        \mathrm{fran}_k(A) \geq C(n,k,2).
    \end{align}
    Furthermore, for all $n$ and $k$ there exists such a matrix for which Inequality~\eqref{eq:frank_bigger_cnk} holds with equality.
\end{proposition}

\begin{proof}
    Let 
    \[
        A = \sum_{\ell=1}^r \mathbf{v}_{\ell}\mathbf{v}_{\ell}^*
    \]
    be a factor-width-$k$ decomposition of $A$. For each $1 \leq \ell \leq r$ let $S_\ell \subseteq [n]$ be the set of indices corresponding to non-zero entries of $\mathbf{v}_{\ell}$. Note that $|S_\ell| \leq k$ for all $\ell$. Since each entry of $A$ is non-zero we know that, for every $i \neq j$, there exists some $\ell$ for which $\{i,j\} \subseteq S_\ell$. In other words, $\{S_1,S_2,\ldots,S_r\}$ is an $(n,k,2)$-covering design, so $r \geq C(n,k,2)$.

    To see that the ``furthermore'' statement holds, suppose $\{S_1,S_2,\ldots,S_r\}$ is an $(n,k,2)$-covering design with $r = C(n,k,2)$. For each $1 \leq \ell \leq r$, let $\mathbf{v}_\ell \in \mathbb{R}^n$ be the vector with $i$-th entry equal to $1$ if $i \in S_\ell$ and $0$ if $i \notin S_\ell$. Then the matrix
    \[
        A = \sum_{\ell=1}^r \mathbf{v}_{\ell}\mathbf{v}_{\ell}^*
    \]
    has factor width at most $k$ and $\mathrm{fran}_k(A) = C(n,k,2)$.
\end{proof}

Numerous results and bounds are known about the quantity $C(n,k,2)$---see \cite{GS07} and the references therein---but its computation in general seems to be difficult. Nevertheless, the bound $C(n,k,2) \geq \lceil (n/k) \lceil (n-1)/(k-1) \rceil \rceil$ of Sch\"{o}nheim \cite{Sch64} is quite close to the best possible (for example, it is an equality when $k \leq 3$ \cite{FH58} or $k = 4$, $n \notin \{7,9,10,19\}$ \cite{Mil72,Mil73}, and for all $k$ it is known to \emph{almost} be attained as long as $n$ is large enough \cite{Gor96}), which leads immediately to the following corollary:

\begin{corollary}\label{cor:fac_wid_covering_design}
    Suppose $A \in \mathcal{M}_n^{+}$ has factor width $k \geq 2$ and all of its entries non-zero. Then
    \[
        \mathrm{fran}_k(A) \geq \left\lceil \frac{n}{k}\left\lceil \frac{n-1}{k-1} \right\rceil \right\rceil.
    \]
\end{corollary}

When the matrix $A \in \M_n^{+}$ has some entries equal to $0$, factor width ranks smaller than $C(n,k,2)$ are possible. However, there is another previously-studied quantity related to coverings that provides a lower bound on the factor width rank in this case \cite{GHH96}:

\begin{definition}\label{defn:edge_coverings}
    Let $k$ be a positive integer. A \emph{$k$-clique covering} of a graph $G$ is a collection of cliques of size $k$ (i.e., \emph{$k$-cliques}) for which every edge of $G$ is contained in at least one of the $k$-cliques. The minimum number of $k$-cliques possible in a $k$-clique covering of $G$ is denoted by $\textup{cc}_k(G)$.
\end{definition}

We note that the $k$-cliques used in a $k$-clique covering of a graph do not need to be contained within the graph $G$: they are allowed to have edges that $G$ does not. For this reason, there is no obvious relationship between the quantity $\textup{cc}_k(G)$ and the usual clique cover number $\textup{cc}(G)$ (which considers cliques of any size, but those cliques must be contained within $G$).

Our interest in $k$-clique coverings comes from the following proposition, which shows that they bound the factor-width-$k$ rank of a matrix:

\begin{proposition}\label{prop:kclique_covering_bound}
    Suppose $A \in \mathcal{M}_n^{+}$ has factor width $k \geq 2$. Let $G$ be the graph on $n$ vertices that contains $\{i,j\}$ as an edge if and only if $a_{i,j} \neq 0$. Then
    \begin{align}\label{eq:frank_bigger_cck}
        \mathrm{fran}_k(A) \geq \textup{cc}_k(G).
    \end{align}
\end{proposition}

\begin{proof}
    Let 
    \[
        A = \sum_{\ell=1}^r \mathbf{v}_{\ell}\mathbf{v}_{\ell}^*
    \]
    be a factor-width-$k$ decomposition of $A$. For each $1 \leq \ell \leq r$ let $G_\ell \subseteq [n]$ be a clique on the vertices corresponding to non-zero entries of $\mathbf{v}_{\ell}$ (if $\mathbf{v}_{\ell}$ has strictly fewer than $k$ non-zero entries then simply add extra vertices and edges to $G_\ell$ arbitrarily until it is a $k$-clique). Whenever $a_{i,j} \neq 0$ there exists some $\ell$ such that the $i$-th and $j$-th entries of $\mathbf{v}_{\ell}$ are both non-zero. In other words, the edge $\{i,j\}$ of $G$ is covered by the $k$-clique $G_\ell$, so $r \geq \textup{cc}_k(G)$.
\end{proof}

We close this subsection by noting that Proposition~\ref{prop:kclique_covering_bound} is a direct generalization of Proposition~\ref{prop:fac_wid_covering_design}: if all of the entries of $A$ are non-zero then its associated graph $G$ is the complete graph $K_n$, and it is straightforward to show that $\textup{cc}_k(K_n) = C(n,k,2)$.

\begin{example}\label{exam:cube_graph}
    Consider the cube graph $Q_3$ on $8$ vertices, illustrated in Figure~\ref{fig:cube_graph}. We consider the factor-width-3 rank $\mathrm{fran}_3(A)$ of a matrix that has zero pattern defined by this cube graph (i.e., $a_{i,j} \neq 0$ exactly when $(i,j)$ is an edge of the $Q_3$).

    \begin{figure}[htb]
        \centering
        \begin{tikzpicture}[scale=2.25]
        \newcounter{vertnum};
        \setcounter{vertnum}{1};
        \foreach \x in {0,1}
            \foreach \y in {0,1}
                \foreach \z in {0,1}{
                    \node[draw, circle, fill=white, draw=black, inner sep=1.5pt] (v\x\y\z) at (\x,\y,\z) {\arabic{vertnum}};
                    \addtocounter{vertnum}{1};
                }
        
        \foreach \x in {0,1}
            \foreach \y in {0,1}
                \draw (v\x\y0) -- (v\x\y1);
        
        \foreach \x in {0,1}
            \foreach \z in {0,1}
                \draw (v\x0\z) -- (v\x1\z);
        
        \foreach \y in {0,1}
            \foreach \z in {0,1}
                \draw (v0\y\z) -- (v1\y\z);
        
        \draw (v000) -- (v001) -- (v011) -- (v010) -- cycle;
        \draw (v100) -- (v101) -- (v111) -- (v110) -- cycle;
        
        \foreach \y in {0,1}
            \foreach \z in {0,1}
                \draw (v0\y\z) -- (v1\y\z);
    \end{tikzpicture}
    \caption{The cube graph $Q_3$.}\label{fig:cube_graph}
    \end{figure}
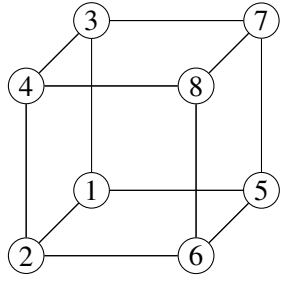

    We claim that $\textup{cc}_3(Q_3) = 6$. To verify this claim, first note that $\textup{cc}_3(Q_3) \geq 6$ since $Q_3$ has no triangles (i.e., $3$-cliques), so each triangle in a $3$-clique covering of $Q_3$ contains at most $2$ of its $12$ edges. On the other hand, we can see that $\textup{cc}_3(Q_3) \leq 6$ by explicitly constructing a $3$-clique covering that makes use of the following 6 triangles (where we simply label the $3$ vertices in each triangle, using the numbering from Figure~\ref{fig:cube_graph}), in the obvious way: 1--2--3, 1--5--7, 2--4--6, 3--4--8, 3--7--8, and 5--6--8.

    Proposition~\ref{prop:kclique_covering_bound} then tells us that every factor-width-3 matrix with non-zero pattern determined by $Q_3$ has factor-width-3 rank at least 6. This is an improvement over the general lower bound of $4$ for the (not-necessarily-factor-width-3) rank of a matrix with non-zero pattern determined by $Q_3$ \cite[Theorem~3.1]{aim2008zero}.
\end{example}

\subsection{Factor Width Rank of Small Matrices}\label{sec:small_mats}

We now summarize exactly what we do and don't know about the factor-width-$k$ rank of small matrices. To start, we note that the case of a $2 \times 2$ matrix $A$ is trivial since $\mathrm{fran}_k(A) = \mathrm{rank}(A)$ when $k \in \{1,n\}$.

For a $3 \times 3$ matrix $A$, the situation is slightly less straightforward, but we can similarly compute factor width rank as follows. We first determine the factor width $k$ via semidefinite programming. If $k \in \{1,3\}$ then $\mathrm{fran}_k(A) = \mathrm{rank}(A)$. On the other hand, if the factor width is $k = 2$ then we split into two cases:
\begin{itemize}
    \item If all off-diagonal entries of $A$ are non-zero then, by Corollary~\ref{cor:fac_wid_rank2_nonzero}, we have $\mathrm{fran}_2(A) = 3$.

    \item If at least one off-diagonal entry of $A$ equals zero then there exists a permutation matrix $P$ for which $PAP^*$ is tridiagonal. It follows from Theorem~\ref{thm:tridiag_fac_two_rank} that
    \[
        \mathrm{fran}_2(A) = \mathrm{fran}_2(PAP^*) = \mathrm{rank}(PAP^*) = \mathrm{rank}(A).
    \]
\end{itemize}

For a $4 \times 4$ matrix, the situation is more complicated. We can again compute the factor width $k$ via semidefinite programming, and note that if $k \in \{1,4\}$ then $\mathrm{fran}_k(A) = \mathrm{rank}(A)$. If $k \in \{2,3\}$ then we split into cases as follows:

\begin{itemize}
    \item If $k = 3$ and at least one off-diagonal entry of $A$ equals zero then there exists a permutation matrix $P$ for which $PAP^*$ has bandwidth $3$, so it follows from Theorem~\ref{thm:tridiag_fac_two_rank} that $\mathrm{fran}_3(A) = \mathrm{rank}(A)$. On the other hand, if all off-diagonal entries of $A$ are non-zero then we do not know of an effective way to compute $\mathrm{fran}_3(A)$; the best bounds we have are
    \[
        3 \leq \mathrm{fran}_3^{\mathbb{F}}(A) \leq \begin{cases}
                10 & \text{if \ $\mathbb{F} = \R$}, \\
                12 & \text{if \ $\mathbb{F} = \C$}.
            \end{cases}
    \]
    by Corollary~\ref{cor:fac_wid_covering_design} and Propositions~\ref{prop:two_rank_ubs} and~\ref{prop:fac_wid_rank_basic_bounds}.
    
    \item If $k = 2$ and all off-diagonal entries of $A$ are non-zero then, by Corollary~\ref{cor:fac_wid_rank2_nonzero}, we have $\mathrm{fran}_2(A) = 6$. If there exists a permutation matrix $P$ for which $PAP^*$ is tridiagonal then Theorem~\ref{thm:tridiag_fac_two_rank} tells us that $\mathrm{fran}_2(A) = \mathrm{rank}(A)$. There are, up to conjugation by permutation matrices, exactly $5$ other patterns that the non-zero entries of $A$ can have:
    \begin{align*}
        & \begin{bmatrix}
            * & * & * & 0 \\
            * & * & * & * \\
            * & * & * & * \\
            0 & * & * & *
        \end{bmatrix}, \quad \begin{bmatrix}
            * & * & 0 & 0 \\
            * & * & * & * \\
            0 & * & * & * \\
            0 & * & * & *
        \end{bmatrix}, \quad \begin{bmatrix}
            * & * & 0 & * \\
            * & * & * & 0 \\
            0 & * & * & * \\
            * & 0 & * & *
        \end{bmatrix}, \quad \begin{bmatrix}
            * & * & * & * \\
            * & * & 0 & 0 \\
            * & 0 & * & 0 \\
            * & 0 & 0 & *
        \end{bmatrix}, \quad \begin{bmatrix}
            * & 0 & 0 & 0 \\
            0 & * & * & * \\
            0 & * & * & * \\
            0 & * & * & *
        \end{bmatrix}.
    \end{align*}
    For the second matrix above, we apply Theorem~\ref{thm:block_diag} to see that $\mathrm{fran}_2(A) = 4$. The fourth matrix is an arrowhead matrix, so Theorem~\ref{thm:arrowhead_facwid2} tells us that $\mathrm{fran}_2(A) = \rank(A)$. The fifth matrix is block diagonal, so $\mathrm{fran}_2(A)$ is equal to the sum of the factor-width-$2$ ranks of its diagonal blocks. We do not know of an effective way to compute $\mathrm{fran}_2(A)$ when $A$ has the zero pattern of the first (pentadiagonal) or third (cyclic tridiagonal) matrix.
\end{itemize}

\section{How Hadamard Products Affect Factor Width}\label{sec:hadamard_product}

We now investigate how the Hadamard (i.e., entrywise) product affects the factor width and factor width rank of matrices. If $A,B \in \M_n$ then we denote their Hadamard product by $A \odot B$. That is, $A \odot B \in \M_n$ is the matrix whose $(i,j)$-entry is $a_{i,j}b_{i,j}$. If $s \geq 1$ (which is not necessarily an integer) then we furthermore let $A^{\odot s}$ denote the $s$-th Hadamard power of $A$, which has $(i,j)$-entry equal to $a_{i,j}^s$.

The well-known Schur product theorem \cite{Sch11} says that if $A$ and $B$ are positive semidefinite then so is $A \odot B$. Furthermore, it is straightforward to show that if $A$ and $B$ have rank~$1$ then so does $A \odot B$. We know show that analogous results hold for factor width and factor width rank:

\begin{theorem}\label{thm:hadamard_different_matrices}
    Suppose $A,B \in \M_n^{+}$ have factor width at most $k$. Then $A \odot B$ has factor width at most $k$. Furthermore, $\mathrm{fran}_k(A \odot B) \leq \mathrm{fran}_k(A)\mathrm{fran}_k(B)$.
\end{theorem}

\begin{proof}
    Write $A=\sum_{i=1}^\ell A_i$ and $B=\sum_{j=1}^m B_j$ with $A_i$ and $B_j$ all PSD rank-1 matrices each having a principal $k\times k$ nonzero block (or smaller) and $\ell = \mathrm{fran}_k(A)$ and $m = \mathrm{fran}_k(B)$. Then $$A\odot B=\sum_{i=1}^{\ell}\sum_{j=1}^m A_i\odot B_j.$$
    Then each $A_i\odot B_j$ term is also PSD with rank~$1$ and is still non-zero on at most a $k \times k$ principal submatrix, so their sum has factor width at most $k$. Since there are $\ell m$ terms in this sum, we furthermore have $\mathrm{fran}_k(A \odot B) \leq \ell m = \mathrm{fran}_k(A)\mathrm{fran}_k(B)$.
\end{proof}

We note that the bound $\mathrm{fran}_k(A \odot B) \leq \mathrm{fran}_k(A)\mathrm{fran}_k(B)$ from the above theorem cannot be improved without taking $k$, $n$, or additional features of $A$ or $B$ into account. In particular, if $A$ and $B$ are random PSD matrices (chosen according to some continuous distribution) with only their top-left $k \times k$ block non-zero, then typically we will have
\[
    \mathrm{fran}_k(A \odot B) = \rank(A \odot B) = \rank(A)\rank(B) = \min\{k,\mathrm{fran}_k(A)\mathrm{fran}_k(B)\}.
\]
For example, if
\[
    A = \begin{bmatrix}
     2 & 1 & 2 & 1 & 0 \\
     1 & 1 & 1 & 1 & 0 \\
     2 & 1 & 2 & 1 & 0 \\
     1 & 1 & 1 & 1 & 0 \\
     0 & 0 & 0 & 0 & 0
    \end{bmatrix} \quad \text{and} \quad B = \begin{bmatrix}
     2 & 0 & 1 & -1 & 0 \\
     0 & 2 & 1 & 1 & 0 \\
     1 & 1 & 1 & 0 & 0 \\
    -1 & 1 & 0 & 1 & 0 \\
     0 & 0 & 0 & 0 & 0
    \end{bmatrix}
\]
then it is straightforward to verify that $\rank(A) = \rank(B) = \mathrm{fran}_4(A) = \mathrm{fran}_4(B) = 2$ and $\rank(A \odot B) = \mathrm{fran}_4(A \odot B) = 4$.

Theorem~\ref{thm:hadamard_different_matrices} can be extended straightforwardly to the Hadamard product of $s \geq 2$ matrices, though the right-hand-side of the resulting bound $\mathrm{fran}_k(A_1 \odot \cdots \odot A_s) \leq \mathrm{fran}_k(A_1)\cdots\mathrm{fran}_k(A_s)$ will quickly exceed the overall maximum possible factor width rank established by Propositions~\ref{prop:two_rank_ubs} and~\ref{prop:fac_wid_rank_basic_bounds}. When all the matrices are the same as each other, however, we get the following slightly better bound:

\begin{theorem}\label{thm:hadamard_same_matrix}
    Let $s \geq 1$ be an integer and suppose $A \in \M_n^{+}$ has factor width at most $k$. Then $A^{\odot s}$ has factor width at most $k$. Furthermore,
    \[
        \mathrm{fran}_k\big(A^{\odot s}\big) \leq \binom{\mathrm{fran}_k(A)+s-1}{s}.
    \]
\end{theorem}

\begin{proof}
    Write $A=\sum_{i=1}^\ell A_i$ with $A_i$ all PSD rank-1 matrices each having a principal $k\times k$ nonzero block (or smaller) and $\ell = \mathrm{fran}_k(A)$. Then
    \begin{align}\begin{split}\label{eq:As_multinomial}
        A^{\odot s} & = \sum_{i_1,i_2,\ldots,i_s=1}^{\ell} A_{i_1}\odot A_{i_2} \odot \cdots \odot A_{i_s} \\
        & = \sum_{1 \leq i_1 \leq i_2 \leq \cdots \leq i_s \leq \ell} \binom{s}{I_1,I_2,\ldots,I_\ell} A_{i_1}\odot A_{i_2} \odot \cdots \odot A_{i_s},
    \end{split}\end{align}
    where $I_j := |\{m : i_m = j\}|$ ($1 \leq j \leq \ell$) is the number of subscripts equal to $j$, so $\binom{s}{I_1,I_2,\ldots,I_\ell}$ is a multinomial coefficient that counts how many times, taking commutativity of the Hadamard product into account, the term $A_{i_1}\odot A_{i_2} \odot \cdots \odot A_{i_s}$ appears in the first sum of Equation~\eqref{eq:As_multinomial}.
    
    Since there are $\binom{\ell+s-1}{s}$ terms in the second sum of Equation~\eqref{eq:As_multinomial}, and each term $A_{i_1}\odot A_{i_2} \odot \cdots \odot A_{i_s}$ of that sum has rank~$1$ and is still non-zero on at most a $k \times k$ principal submatrix, the result follows.
\end{proof}

In particular, Theorem~\ref{thm:hadamard_same_matrix} shows that if $\mathrm{fran}_k(A)$ is fixed then $\mathrm{fran}_k\big(A^{\odot s}\big)$ grows at most as a polynomial in $s$, not as an exponential in $s$ like when computing Hadamard products of different matrices as in Theorem~\ref{thm:hadamard_different_matrices}.

\subsection{Non-Integer Hadamard Powers}

Theorem~\ref{thm:hadamard_same_matrix} tells us that positive integer Hadamard powers of matrices with small factor width still have small factor width. Determining whether or not the same is true of non-integer powers is much trickier. In fact, even making sense of this question is trickier, since non-integer powers of negative numbers can be complex. For this reason, we now restrict our attention to matrices with all of their entries non-negative.\footnote{While factor width can be extended to complex matrices (see Section~\ref{sec:real_v_complex}), there are additional problems that arise when considering positive semidefiniteness of non-integer Hadamard powers of matrices with negative entries; see \cite[Section~3]{FH1977}.}

The corresponding question for positive semidefiniteness (i.e., ``Which non-integer Hadamard powers preserve positive semidefiniteness?'') is answered by the Fitzgerald--Horn theorem \cite[Theorem~2.2]{FH1977}: if $A \in \M_n^+$ and $s > 0$ then we can conclude that that $A^{\odot s} \in M_n^+$ if and only if $s$ is and integer or $s \geq n-2$. We conjecture the following generalization of the Fitzgerald--Horn theorem for factor width:

\begin{conjecture}\label{conj:main_hadamard}
    Let $s > 0$ and suppose $A \in \M_n^{+}$ is entrywise non-negative with factor width at most $k$. If $s \geq \min\{k-1,n-2\}$ then $A^{\odot s}$ also has factor width at most $k$. Conversely, if $s < \min\{k-1,n-2\}$ is not an integer then there exists a factor-width-$k$ matrix $A$ for which $A^{\odot s}$ is not positive semidefinite.
\end{conjecture}

It is perhaps worth noting a few special cases in which Conjecture~\ref{conj:main_hadamard} is known to hold:

\begin{itemize}
    \item If $k = n$ then the conjecture is simply the Fitzgerald--Horn theorem, and is thus true.

    \item If $k = 1$ then the conjecture says that all positive Hadamard powers of diagonal positive semidefinite matrices are still diagonal and positive semidefinite, which is trivially true.

    \item If $A \in \M_n^{+}$ has bandwidth $1 \leq k < n$ then (in light of Theorem~\ref{thm:factor_width_and_bandwidth}) Conjecture~\ref{conj:main_hadamard} conjectures that $A^{\odot s}$ must still be positive semidefinite exactly when $s \geq k-1$ or $s$ is an integer. This was proved in \cite[Theorem~1.4]{chordalpositivity}. The tridiagonal (i.e., bandwidth $k = 2$) case was also proved in \cite{2022positivity}.

    \item The ``conversely'' statement of Conjecture~\ref{conj:main_hadamard} is known to hold: it follows from the fact that if $s < \min\{k-1,n-2\}$ is not an integer then there is a positive semidefinite matrix with bandwidth $k$ (and thus factor width $k$, by Theorem~\ref{thm:factor_width_and_bandwidth}) for which $A^{\odot s}$ is not positive semidefinite \cite[Theorem~1.4]{chordalpositivity}.
\end{itemize}

\noindent We now prove one additional case of Conjecture~\ref{conj:main_hadamard}: the $k = 2$ case. We first need a lemma.

\begin{lemma}\label{lem:DAD_s_entries}
    Let $A\in M_n^{+}$ be an entrywise non-negative matrix, $D$ be an invertible diagonal matrix, and $s>0$ be a real number. Then $(DAD)^{\odot s} = D^s A^{\odot s}D^s$.
\end{lemma}

\begin{proof}
   Just notice that, for all $i$ and $j$, we have $[(DAD)^{\odot s}]_{i,j} = a_{i,j}^s d_i^s d_j^s = [D^s A^{\odot s}D^s]_{i,j}$.
\end{proof}

\begin{theorem}\label{thm:factor_width_two_had_powers}
    Let $s > 0$ and suppose $A \in \M_n^{+}$ ($n \geq 3$) is entrywise non-negative with factor width at most $2$. If $s \geq 1$ then $A^{\odot s}$ also has factor width at most $2$. Conversely, if $s < 1$ then there exists a factor-width-$2$ matrix $A$ for which $A^{\odot s}$ is not positive semidefinite.
\end{theorem}

\begin{proof}
    As noted earlier, the ``conversely'' statement is already known, so we just prove if $s \geq 1$ and $A \in \M_n^{+}$ ($n \geq 3$) is entrywise non-negative with factor width at most $2$ then so is $A^{\odot s}$. To this end, recall from Lemma~\ref{lem:SDD} that there exists an invertible diagonal matrix $D \in \M_n^+$ for which $B := DAD$ is diagonally dominant. Then for all $1 \leq j \leq n$ we have
    \begin{align*}
        b_{j,j}^s & \geq \left(\sum_{i \neq j} |b_{i,j}|\right)^s \geq \sum_{i \neq j} |b_{i,j}|^s,
    \end{align*}
    where the second inequality comes from the fact that the vector $1$-norm is at least as large as the vector $s$-norm when $s \geq 1$ (i.e., $\|(|b_{1,j}|,\ldots,|b_{n,j}|)\|_1 \geq \|(|b_{1,j}|,\ldots,|b_{n,j}|)\|_s$). It follows that $B^{\odot s} = (DAD)^{\odot s} = D^s A^{\odot s}D^s$ is also diagonally dominant, where the second equality follows from Lemma~\ref{lem:DAD_s_entries}. Applying Lemma~\ref{lem:SDD} again now reveals that $A^{\odot s} = D^{-s}B^{\odot s}D^{-s}$ has factor width at most $2$, completing the proof.
\end{proof}

\subsection{Large Hadamard Powers}

The earlier results in this section (Theorem~\ref{thm:hadamard_same_matrix} in particular) demonstrate that factor width does not increase when taking Hadamard powers. In this final subsection, we show that in fact the factor width typically decreases drastically when taking Hadamard powers. In particular, we show that for all positive semidefinite matrices outside of a set of measure $0$, there is a sufficiently large Hadamard power that has factor width $2$:

\begin{theorem}\label{thm:goes_to_fac_wid_2}
    Let $A\in M_n^+$ be an entrywise non-negative PSD matrix. If every principal $2\times 2$ submatrix of $A$ is positive definite, then there exists a natural number $M$ such that $A^{\odot m}$ has factor width at most~$2$ whenever $m \geq M$ is a natural number.
\end{theorem}

\begin{proof}
    Assume without loss of generality that the diagonal entries of $A$ are all non-zero and then pick an invertible diagonal matrix $D$ for which the matrix $B := DAD$ has diagonal entries equal to $1$. By Lemma~\ref{lem:DAD} we know that $A$ and $B$ have the same factor width. Furthermore, since each $2 \times 2$ principal submatrix of $A$ is positive definite, the same is true of $B$, so we have $|b_{i,j}| < 1$ whenever $i \neq j$.
    
    It follows that $\lim_{m\rightarrow \infty}B^{\odot m} = I$, so $\lim_{m\rightarrow \infty}\lambda_{\mathrm{max}}(B^{\odot m}) = 1$. It follows from \cite[Theorem~5]{johnston2022absolutely} that $B^{\odot m}$ has factor width at most~$2$ whenever $m$ is sufficiently large.\footnote{To use that theorem, first rescale $B^{\odot m}$ so that $\tr(B^{\odot m}) = 1$ and then notice that $\lim_{m\rightarrow \infty}\lambda_{\mathrm{max}}(B^{\odot m}) = 1/n$, so $\lambda_{\mathrm{max}}(B^{\odot m}) \leq 1/(n-1)$ when $m$ is sufficiently large.} By Lemma~\ref{lem:DAD_s_entries} we know that $B^{\odot m} = D^m A^{\odot m}D^m$, so $A^{\odot m}$ has the same factor width as $B^{\odot m}$, completing the proof.
\end{proof}

The above theorem is optimal in at least two senses. First, the cone of matrices with factor width at most $2$ is the smallest factor width cone that we could possibly hope large Hadamard powers to land inside of, since the cone of matrices with factor width $1$ consists of just the diagonal PSD matrices. Second, the hypothesis that every $2 \times 2$ principal submatrix is positive definite is required, as demonstrated by the all-ones matrix $J \in \M_n^{+}$ which has the property that $J^{\odot m}$ has factor width $n$ for all $m$.

\section*{Acknowledgements}
    The authors thank Rob~Pratt for introducing them to covering designs~\cite{MO461882}, and the anonymous referee for their insightful comments leading to Theorem~\ref{thm:chordal_main}. N.J.\ was supported by NSERC Discovery Grant number RGPIN-2022-04098. S.P.\ was supported by NSERC Discovery Grant number RGPIN-2019-05276, the Canada Research Chairs Program grant number 101062, and the  Canada Foundation for Innovation grant number 43948. 

\bibliographystyle{alpha}
\bibliography{ref}

\end{document}